\documentclass{amsart}%
\usepackage[initials]{amsrefs}
\usepackage{amssymb,amsmath,amsfonts,latexsym,amsthm,geometry}
\usepackage{graphicx}
\usepackage{amsmath}
\usepackage{amsfonts}
\usepackage{amssymb}%
\providecommand{\U}[1]{\protect\rule{.1in}{.1in}}
\newtheorem{theorem}{Theorem}[section]
\newtheorem{proposition}[theorem]{Proposition}
\newtheorem{corollary}[theorem]{Corollary}

\newtheorem{remark}[theorem]{Remark}

\newtheorem{lemma}[theorem]{Lemma}

\def\aai{a_{\mathbf i}}
\def\hatik{\widehat{i_k}}

\def\lp1n{\ell_{p_1}^n}
\def\lpln{\ell_{p_l}^n}

\newtheorem*{DEFSP}{Theorem - Defant/Sevilla-Peris}
\newtheorem*{PP}{Theorem - Praciano-Pereira}

\numberwithin{equation}{section}
\begin{document}
\title[Sharp generalizations of the multilinear Bohnenblust--Hille inequality]{Sharp generalizations of the multilinear Bohnenblust--Hille inequality}
\date{}
\author[Albuquerque]{N. Albuquerque}
\address{Departamento de Matem\'{a}tica, \newline \indent
Universidade Federal da Para\'{i}ba, \newline \indent
58.051-900 - Jo\~{a}o Pessoa, Brazil.}
\email{ngalbqrq@gmail.com}

\author[Bayart]{F. Bayart}
\address{Laboratoire de Math\'ematique, \newline \indent
Universit\'e Blaise Pascal Campus des C\'ezeaux, \newline \indent
F-63177 Aubiere Cedex, France.}
\email{Frederic.Bayart@math.univ-bpclermont.fr}

\author[Pellegrino]{D. Pellegrino}
\address{Departamento de Matem\'{a}tica, \newline \indent
Universidade Federal da Para\'{i}ba, \newline \indent
58.051-900 - Jo\~{a}o Pessoa, Brazil.}
\email{dmpellegrino@gmail.com and pellegrino@pq.cnpq.br}

\author[Seoane]{J.B. Seoane-Sep\'{u}lveda}
\address{Departamento de An\'{a}lisis Matem\'{a}tico,\newline\indent Facultad de Ciencias Matem\'{a}ticas, \newline\indent Plaza de Ciencias 3, \newline\indent Universidad Complutense de Madrid,\newline\indent Madrid, 28040, Spain.}
\email{jseoane@mat.ucm.es}
\keywords{Bohnenblust--Hille inequality; Littlewood's $4/3$ inequality; Absolutely
summing operators}

\begin{abstract}
We prove that the multilinear Bohnenblust--Hille is a particular case of a
quite general family of optimal inequalities.

\end{abstract}
\maketitle

%\thanks{D. Pellegrino was supported by CNPq Grant 477124/2012-7, INCT-Matem\'{a}tica and CAPES-NF.}

\section{Introduction}

The Bohnenblust-Hille inequality (see \cite{bh}) asserts that for all positive
integers $m\geq2$ there is a constant $C=C(m)\geq1$ such that
\begin{equation}
\left(  \sum\limits_{i_{1},...,i_{m}=1}^{\infty}\left\vert A(e_{i_{^{1}}%
},...,e_{i_{m}})\right\vert ^{\frac{2m}{m+1}}\right)  ^{\frac{m+1}{2m}}\leq
C\left\Vert A\right\Vert \label{juui}%
\end{equation}
for all continuous $m$-linear forms $A:c_{0}\times\cdots\times c_{0}%
\rightarrow\mathbb{K}$. Here, as usual, $\mathbb{K}$ denotes the fields of
real or complex scalars. This inequality has important applications in various
fields of analysis and mathematical physics (\cite{boas, monta}). The case
$m=2$ is the well-known Littlewood's $4/3$ inequality \cite{littlewood}.

In this paper we show that the Bohnenblust--Hille inequality is a very
particular case of a large family of sharp inequalities. More precisely, we
prove the following general result:

\begin{theorem}
\label{THMBHQ} Let $m\geq1$, let $q_{1},...,q_{m}\in\lbrack1,2].$ The
following assertions are equivalent:

(1) There is a constant $C_{q_{1}...q_{m}}\geq1$ such that{\small {
\[
\left(  {\textstyle\sum\limits_{i_{1}=1}^{\infty}}\left(  {\textstyle\sum
\limits_{i_{2}=1}^{\infty}}\left(  ...\left(  {\textstyle\sum\limits_{i_{m-1}%
=1}^{\infty}}\left(  {\textstyle\sum\limits_{i_{m}=1}^{\infty}}\left\vert
A\left(  e_{i_{1}},...,e_{i_{m}}\right)  \right\vert ^{q_{m}}\right)
^{\frac{q_{m-1}}{q_{m}}}\right)  ^{\frac{q_{m-2}}{q_{m-1}}}\cdots\right)
^{\frac{q_{2}}{q_{3}}}\right)  ^{\frac{q_{1}}{q_{2}}}\right)  ^{\frac{1}%
{q_{1}}}\leq C_{q_{1}\ldots q_{m}}\left\Vert A\right\Vert
\]
}}for all continuous $m$-linear forms $A:c_{0}\times\cdots\times
c_{0}\rightarrow\mathbb{K}$.

(2) $\frac{1}{q_{1}}+\cdots+\frac{1}{q_{m}}\leq\frac{m+1}{2}.$
\end{theorem}

The Bohnenblust--Hille inequality is just the particular case%
\[
q_{1}=\cdots=q_{m}=\frac{2m}{m+1}.
\]

This kind of inequalities was already considered in \cite{mona} when $m=2$, as
generalizations of Littlewood's 4/3 inequality. The strategy for the proof of
$(2)\implies(1)$ in Theorem \ref{THMBHQ} will be very simple, maybe simpler
than all previous known proofs of the Bohnenblust-Hille inequality. The
starting point is the generalized Littlewood mixed $(\ell^{1},\ell^{2})-$norm
inequality, that is that Theorem \ref{THMBHQ} is true when $(q_{1},\dots
,q_{m})=(1,2,\dots,2)$. This property is well-known and it is a consequence of
Khintchine's inequality. Using this and nothing else than Minkowski's
inequality and interpolation, we will deduce the general case.

Several generalizations of the Bohnenblust-Hill inequality have already been
obtained. Let us quote two of them. The first one is a vector-valued version
due to Defant and Sevilla-Peris in \cite{def1}.

\begin{DEFSP}
Let $m\geq1$, $1\leq s\leq q\leq2$. Define
\[
\rho=\frac{2m}{m+2\left(  \frac1s-\frac1q\right)  }.
\]
Then there exists a constant $C>0$ such that, for every continuous $m$-linear
mapping $A:c_{0}\times\dots\times c_{0}\to\ell_{s}$, then
\[
\left(  \sum_{i_{1},\dots,i_{m}=1}^{+\infty}\|A(e_{i_{1}},\dots,e_{i_{m}%
})\|^{\rho}_{q}\right)  ^{1/\rho}\leq C\|A\|.
\]

\end{DEFSP}

Second, in the spirit of the Hardy-Littlewood generalization of Littlewood's
4/3 inequality, Praciano-Pereira has studied in \cite{PP80} the effect of
replacing $c_{0}$ by $\ell_{p}$ in the Bohnenblust-Hille inequality. For the
sake of convenience, let us introduce the following notations which will be
used throughout the paper: for $\mathbf{p}=(p_{1},\dots,p_{m})\in
[1,+\infty]^{m}$, let
\[
\left|  {\frac{1}{\mathbf{p}}}\right|  =\frac{1}{p_{1}}+\dots+\frac{1}{p_{m}%
}.
\]
For $p\geq1$, let also $X_{p}=\ell_{p}$ and let us define $X_{\infty}=c_{0}$.

\begin{PP}
Let $m\geq1$, $1\leq p_{1},\dots,p_{m}\leq+\infty$ with $\left|  {\frac
{1}{\mathbf{p}}}\right|  \leq\frac12$. Define
\[
\rho=\frac{2m}{m+1-2\left|  {\frac{1}{\mathbf{p}}}\right|  }.
\]
Then there exists a constant $C>0$ such that, for every continuous $m$-linear
mapping $A:X_{p_{1}}\times\dots\times X_{p_{m}}\to\mathbb{C}$,
\[
\left(  \sum_{i_{1},\dots,i_{m}=1}^{+\infty}|A(e_{i_{1}},\dots,e_{i_{m}%
})|^{\rho}\right)  ^{1/\rho}\leq C\|A\|.
\]

\end{PP}

Both approaches can be embedded into our framework, so that we are able to
prove the following result.

\begin{theorem}
\label{THMMAIN} Let $m\geq1$, let $1\leq s\leq q\leq2$, let $\mathbf{p}%
=(p_{1},\dots,p_{m})\in\lbrack1,+\infty]^{m}$ such that
\[
\frac{1}{s}-\frac{1}{q}-\left\vert {\frac{1}{\mathbf{p}}}\right\vert \geq0.
\]
Let also $q_{1},\dots,q_{m}\in\lbrack\lambda,2],$ where
\[
\lambda=\frac{1}{\frac{1}{2}+\frac{1}{s}-\frac{1}{q}-\left\vert {\frac
{1}{\mathbf{p}}}\right\vert }.
\]
Then the following are equivalent:

(1) There is a constant $C_{q_{1}...q_{m}}\geq1$ such that{\small {
\[
\left(  {\textstyle\sum\limits_{i_{1}=1}^{\infty}}\left(  {\textstyle\sum
\limits_{i_{2}=1}^{\infty}}\left(  ...\left(  {\textstyle\sum\limits_{i_{m-1}%
=1}^{\infty}}\left(  {\textstyle\sum\limits_{i_{m}=1}^{\infty}} \|A\left(
e_{i_{1}},...,e_{i_{m}}\right)  \|_{q} ^{q_{m}}\right)  ^{\frac{q_{m-1}}%
{q_{m}}}\right)  ^{\frac{q_{m-2}}{q_{m-1}}}\cdots\right)  ^{\frac{q_{2}}%
{q_{3}}}\right)  ^{\frac{q_{1}}{q_{2}}}\right)  ^{\frac{1}{q_{1}}}\leq
C_{q_{1}\ldots q_{m}}\left\Vert A\right\Vert
\]
}}for all continuous $m$-linear forms $A:X_{p_{1}}\times\cdots\times X_{p_{r}%
}\rightarrow\ell_{s}$.

(2) $\frac{1}{q_{1}}+\cdots+\frac{1}{q_{m}}\leq\frac{m}{2}+\frac
1s-\frac1q-\left|  {\frac{1}{\mathbf{p}}}\right|  .$
\end{theorem}

When all the $q_{i}$ are equal, we get the following corollary:

\begin{corollary}
Let $m\geq1$, let $1\leq s\leq q\leq2$, let $\mathbf{p}=(p_{1},\dots,p_{m}%
)\in[1,+\infty]^{m}$ such that
\[
\frac1s-\frac1q-\left|  {\frac{1}{\mathbf{p}}}\right|  \geq0.
\]
Let us define
\[
\rho=\frac{2m}{m+2\left(  \frac1s-\frac1q-\left|  {\frac{1}{\mathbf{p}}%
}\right|  \right)  }.
\]
Then there exists a constant $C>0$ such that, for every continuous $m$-linear
mapping $A:X_{p_{1}}\times\dots\times X_{p_{m}}\to\ell_{s}$, then
\[
\left(  \sum_{i_{1},\dots,i_{m}=1}^{+\infty}\|A(e_{i_{1}},\dots,e_{i_{m}%
})\|^{\rho}_{q}\right)  ^{1/\rho}\leq C\|A\|.
\]

\end{corollary}

It is plain that this corollary extends Defant and Sevilla-Peris result to the
$\ell_{p}$-case (and we get the same result if we choose $p_{1}=\dots
=p_{m}=\infty$). To show that Theorem \ref{THMMAIN} also implies Theorem
\ref{THMBHQ} and the result of Praciano-Pereira, it suffices to choose $q=2$,
$s=1$ and to consider only $m$-linear mappings which have their range in the
span of the first basis vector.

\section{A general interpolation method}

To prove our main theorem, we shall prove several particular cases and then
deduce the result by interpolation. To avoid boring technicalities, let us
introduce the following notation: let $Y$ be a Banach space and let
$\mathbf{q}=(q_{1},\dots,q_{m})\in[1,+\infty)^{m}$. Then $\ell_{\mathbf{q}%
}(Y)$ denotes the Lorentz space $\ell_{q_{1}}(\ell_{q_{2}}(\dots(\ell_{q_{m}%
}(Y))\dots)).$

\begin{proposition}
Let $Z,Y$ be Banach spaces, let $N,m\geq1$ and, for any $k\in\{1,\dots,N\}$,
let $\mathbf{q}(k)\in[1,+\infty)^{m}$. Suppose that, for any $k\in
\{1,\dots,N\}$, $T$ is a bounded operator from $Z$ into $\ell_{q(k)}(Y)$ with
norm less than $C$. Then, for any $\mathbf{q}\in(1,+\infty)^{m}$ such that
$\left(  \frac1{q_{1}},\dots,\frac1{q_{m}}\right)  $ belongs to the convex
hull of $\left(  \frac{1}{q_{1}(k)},\dots,\frac1{q_{m}(k)}\right)  $,
$k=1,\dots,N$, $T$ is a bounded operator from $Z$ to $\ell_{\mathbf{q}}(Y)$
with norm less than $C$.
\end{proposition}

\begin{proof}
The proof is done by induction on $N$, using that, for any $\theta\in[0,1]$,
\[
[\ell_{\mathbf{p}}(Y),\ell_{\mathbf{q}}(Y)]_{\theta}=\ell_{\mathbf{r}}(Y)
\]
with $\frac{1}{r_{1}}=\frac{\theta}{p_{i}}+\frac{1-\theta}{q_{i}}$ (see
\cite{berg}). The inductive step follows by associativity. Indeed, if we may
write
\[
\frac{1}{q_{i}}=\frac{\theta_{1}}{q_{i}(1)}+\dots+\frac{\theta_{N}}{q_{i}(N)}%
\]
with $\theta_{1}+\dots+\theta_{N}=1$, then we also have
\[
\frac{1}{q_{i}}=\frac{1-\theta_{N}}{p_{i}}+\frac{\theta_{N}}{q_{i}(N)}%
\]
setting
\[
\frac1{p_{i}}=\frac{\alpha_{1}}{q_{i}(1)}+\dots+\frac{\alpha_{N-1}}%
{q_{i}(N-1)},\ \alpha_{i}=\frac{\theta_{i}}{1-\theta_{N}}%
\]
so that $\alpha_{1}+\dots+\alpha_{N-1}=1$.
\end{proof}

The previous proposition points out the necessity to describe certain convex hulls.

\begin{lemma}
Let $m\geq2$, $0<a<b$ and let, for any $k\in\{1,\dots,m\}$, $M_{k}%
=(a,\dots,a,b,a\dots,a)$ where $b$ is at the $k$-th position. Then the convex
hull of $M_{1},\dots,M_{m}$ is the set of points $(x_{1},\dots,x_{m})$ such
that $x_{1}+\dots+x_{m}=(m-1)a+b$ and $x_{k}\in[a,b]$ for any $k\geq1$.
\end{lemma}

\begin{proof}
This follows from an easy induction.
\end{proof}

As a consequence of these two results, we get the following corollary.

\begin{corollary}
Let $Z,Y$ be Banach spaces, let $m\geq2$ and let $1\leq\lambda\leq2$. For any
$k\in\{1,\dots,m\}$, let $\mathbf{q}(k)=(2,\dots,2,\lambda,2,\dots,2)$ where
$\lambda$ is at the $k$-th position. Suppose that $T$ maps continuously $Z$
into $\ell_{\mathbf{q}(k)}(Y)$ with norm less than $C$. Then, for any
$\mathbf{q}\in[\lambda,2]^{m}$ with
\[
\frac1{q_{1}}+\dots+\frac{1}{q_{m}}=\frac{m-1}2+\frac1{\lambda},
\]
$T$ maps continuously $Z$ into $\ell_{\mathbf{q}}(Y)$.
\end{corollary}

\section{Applications of the Minkowski's inequality}

In this section, we shall prove the following proposition, which shows that an
$(\ell_{\lambda},\ell_{2})$-mixed norm inequality implies many other
inequalities of this kind.

\begin{proposition}
\label{PROPMINKO} Let $Y$ be a Banach space, let $m\geq1$ and let $\lambda
\in[1,2]$. For any $\mathbf{i}\in\{1,\dots,m\}^{\mathbb{N}}$, let
$x_{\mathbf{i}}\in Y$. Define, for any $k\in\{1,\dots,m\}$, $\mathbf{q}%
(k)=(2,\dots,2,\lambda,2,\dots,2)$. Assume that, for any $\sigma$ a
permutation of $\{1,\dots,m\}$, $x_{\sigma(\mathbf{i})}\in\ell_{\mathbf{q}%
(1)}(Y)$. Then, for any $k\in\{1,\dots,m\}$ and for any permutation $\sigma$
of $\{1,\dots,m\}$, $x_{\sigma(\mathbf{i})}\in\ell_{\mathbf{q}(k)}(Y)$ with
\[
\|x_{\sigma(\mathbf{i})}\|\leq\max_{\tau}\|x_{\tau(\mathbf{i})}\|_{\ell
_{\mathbf{q}(1)}(Y)}.
\]

\end{proposition}

Before to proceed with the proof, let us explain the notations. To say that
$x_{\sigma(\mathbf{i})}$ belongs to $\ell_{\mathbf{q}(1)}(Y)$ simply means
that
\[
\left(  {\sum\limits_{i_{\sigma(1)}=1}^{\infty}}\left(  {\sum
\limits_{i_{\sigma(2)}=1}^{\infty}}\left(  ...\left(  {\sum\limits_{i_{\sigma
(m-1)}=1}^{\infty}}\left(  {\sum\limits_{i_{\sigma(m)}=1}^{\infty}}\left\Vert
x_{\mathbf{i}}\right\Vert ^{2}\right)  ^{\frac{2}{2}}\right)  ^{\frac{2}{2}%
}\cdots\right)  ^{\frac{2}{2}}\right)  ^{\frac{\lambda}{2}}\right)  ^{\frac
{1}{\lambda}}<+\infty.
\]
Since there are many $2/2$ which simplify, we shall also denote the above
quantity
\[
\left(  \sum_{i_{k}=1}^{+\infty}\left(  \sum_{\widehat{i_{k}}}\Vert
x_{\mathbf{i}}\Vert^{2}\right)  ^{\lambda/2}\right)  ^{1/\lambda}<+\infty,
\]
$\widehat{i_{k}}$ meaning that the sum is over all coordinates except the
$k$-th coordinate.

\begin{proof}
We proceed by induction on $k$ and assume that the property is true until rank
$k-1$. Since the assumptions are invariant by permutation, we may assume
$\sigma=Id$. Then we just need to consider
\[
\alpha=\left(  {\sum\limits_{i_{1}=1}^{\infty}}\dots\sum\limits_{i_{k-1}%
=1}^{\infty}\left(  {\sum\limits_{i_{k}=1}^{\infty}}\left(  {\sum
\limits_{i_{k+1},\dots,i_{m}=1}^{\infty}}\left\Vert x_{\mathbf{i}}\right\Vert
^{2}\right)  ^{\frac{\lambda}{2}}\right)  ^{\frac{2}{\lambda}}\cdots\right)
^{\frac{1}{2}}.%
\]
The integers $i_{1},\dots,i_{k-1}$ being fixed, we set
\[
a_{i_{k-1},i_{k}}=\left(  {\sum\limits_{i_{k+1},\dots,i_{m}=1}^{\infty}%
}\left\Vert x_{\mathbf{i}}\right\Vert ^{2}\right)  ^{\frac{\lambda}{2}}.
\]
By Minkowski's inequality applied in $\ell_{2/\lambda}$,
\[
\sum_{i_{k-1}=1}^{+\infty}\left(  \sum_{i_{k}=1}^{+\infty}a_{i_{k-1},i_{k}%
}\right)  ^{2/\lambda}\leq\left(  \sum_{i_{k}=1}^{+\infty}\left(
\sum_{i_{k-1}=1}^{+\infty}a_{i_{k-1},i_{k}}^{2/\lambda}\right)  ^{\lambda
/2}\right)  ^{2/\lambda}.
\]
Hence,
\[
\alpha\leq\left(  {\sum\limits_{i_{1}=1}^{\infty}}\dots\sum\limits_{i_{k-2}%
=1}^{\infty}\left(  {\sum\limits_{i_{k}=1}^{\infty}}\left(  \sum_{i_{k-1}%
=1}^{+\infty}{\sum\limits_{i_{k+1},\dots,i_{m}=1}^{\infty}}\left\Vert
x_{\mathbf{i}}\right\Vert ^{2}\right)  ^{\frac{\lambda}{2}}\right)  ^{\frac
{2}{\lambda}}\cdots\right)  ^{\frac{1}{2}}.
\]
By the induction hypothesis, applied to the transposition $(k\quad k-1)$,
$\alpha\leq\max_{\tau}\Vert x_{\tau(\mathbf{i})}\Vert_{\ell_{\mathbf{q}(1)}(Y)}$.
\end{proof}

\section{A mixed $(\ell_{\lambda},\ell_{2})$-norm inequality}

We now give the main technical tool towards the proof of Theorem
\ref{THMMAIN}. It is a vector-valued version of the mixed $(\ell_{1},\ell
_{2})-$norm inequality of Littlewood, allowing moreover $\ell_{p}$-spaces
instead of $\ell_{\infty}$. If $Y$ is a Banach space of cotype 2, then
$C_{2}(Y)$ denotes its cotype 2 constant.

\begin{proposition}
\label{PROPL1L2} Let $X,Y$ be Banach spaces where $Y$ has cotype 2. Let
$m,n,r\geq1$, let $\mathbf{p}\in[1,+\infty]^{m}$, let $A:\lp1n\times
\dots\times\ell_{p_{m}}^{n}\to X$ be $m$-linear and let $v:X\to Y$ be an
$(r,1)$-summing operator. Let us assume that
\[
\left\{
\begin{array}
[c]{l}%
\displaystyle \frac1{p_{1}}+\dots+\frac1{p_{m}}<\frac1r\\
\text{For any }l\geq1,\ \sum_{k\neq l}\frac1{p_{k}}\leq\frac1r-\frac12.
\end{array}
\right.
\]
Let us set
\[
\lambda=\frac{1}{\frac1r-\left|  {\frac{1}{\mathbf{p}}}\right|  }.
\]
Then, for any $k\in\{1,\dots,m\}$,
\[
\left(  \sum_{i_{k}}\left(  \sum_{\widehat{i_{k}}}\|vA(e_{i_{1}}%
,\dots,e_{i_{m}}\|^{2}\right)  ^{\frac\lambda2}\right)  ^{\frac1\lambda}%
\leq\big(\sqrt2C_{2}(Y)\big)^{m-1}\pi_{r,1}(v)\|A\|.
\]

\end{proposition}

\begin{proof}
For the sake of convenience, we shall denote $M=(\sqrt2C_{2}(Y))^{m-1}%
\pi_{r,1}(v)$ and $a_{\mathbf{i}}=\|vA(e_{i_{1}},\dots,e_{i_{m}})\|$. We shall
prove by induction on $l\in[0,m-1]$ the following statement $\mathcal{P}_{l}$:

For any $p_{1},\dots,p_{l}\in[1,+\infty]^{l}$ with $\frac1{p_{1}}+\dots
+\frac1{p_{l}}\leq\frac1r-\frac12$, for any $A:\lp1n\times\dots\times
\ell_{p_{l}}^{n}\times\ell_{\infty}^{n}\times\dots\times\ell_{\infty}^{n}\to
X$ $m$-linear, setting
\[
\lambda=\frac{1}{\frac1r-\left(  \frac1{p_{1}}+\dots+\frac1{p_{l}}\right)  },
\]
for any $k\in\{1,\dots,m\}$,
\begin{align}
\left(  \sum_{i_{k}}\left(  \sum_{\widehat{i_{k}}}a_{\mathbf{i}}^{2}\right)
^{\lambda/2}\right)  ^{1/\lambda}\leq M\|A\|. \label{EQ1}%
\end{align}

$\mathcal{P}_{0}$ is proved in \cite{def1}*{Lemma 2}. Let us assume that
$\mathcal{P}_{l-1}$ is true and let us prove $\mathcal{P}_{l}$. For $k\leq l$,
let us fix $x\in\ell_{p_{k}}^{n}$ and let us consider
\begin{align*}
B_{k}:\lp1n\times\dots\times\ell_{\infty}^{n}\times\dots\times\ell_{p_{l}}%
^{n}\times\ell_{\infty}^{n}\times\dots\times\ell_{\infty}^{n}  &
\rightarrow\mathbb{K}\\
(z^{(1)},\dots,z^{(k)},\dots,z^{(m)})  &  \mapsto A(z^{(1)},\dots
,xz^{(k)},\dots,z^{(m)}).
\end{align*}
By applying the induction hypothesis to $B_{k}$, we know that
\[
\left(  \sum_{i_{k}}|x_{i_{k}}|^{\lambda^{\prime}}\left(  \sum_{\widehat
{i_{k}}}a_{\mathbf{i}}^{2}\right)  ^{\lambda^{\prime}/2}\right)
^{1/\lambda^{\prime}}\leq M\Vert A\Vert\Vert x\Vert_{\ell_{p_{k}}^{n}}%
\]
where we have set
\[
\frac{1}{\lambda^{\prime}}=\frac{1}{r}-\sum_{\substack{j=1\\j\neq k}}^{l}%
\frac{1}{p_{j}}.
\]
We optimize this with respect to $x$ describing the unit ball of $\ell_{p_{k}%
}^{n}$. We get
\[
\left(  \sum_{i_{k}}\left(  \sum_{\widehat{i_{k}}}a_{\mathbf{i}}^{2}\right)
^{\lambda^{\prime}\left(  \frac{p_{k}}{\lambda^{\prime}}\right)  ^{\ast}%
}\right)  ^{\frac{1}{\lambda^{\prime}}\times\frac{1}{\left(  \frac{p_{k}%
}{\lambda^{\prime}}\right)  ^{\ast}}}\leq M\Vert A\Vert
\]
where $\left(  \frac{p_{k}}{\lambda^{\prime}}\right)  ^{\ast}$ denotes the
conjugate exponent of $\frac{p_{k}}{\lambda^{\prime}}$. Now,
\[
\frac{1}{\left(  \frac{p_{k}}{\lambda^{\prime}}\right)  ^{\ast}}%
=1-\frac{\lambda^{\prime}}{p_{k}}=\frac{\lambda^{\prime}}{\lambda}.
\]
This shows that (\ref{EQ1}) holds when $k\leq l$. Let us also show that it
also holds when $k>l$. The previous simple trick does not hold and the
induction hypothesis applied to $B_{l}$ and with
\[
\frac{1}{\lambda^{\prime}}=\frac{1}{r}-\sum_{j=1}^{l-1}\frac{1}{p_{j}}%
\]
now gives
\begin{eqnarray}\label{EQ2}
 \forall x\in B_{\lpln},\ \sum_{i_k}\left(\sum_{\hatik}\aai^2|x_{i_l}|^2\right)^{\lambda'/2}\leq M^{\lambda'}\|A\|^{\lambda'}.
\end{eqnarray}
Let us denote, for $i_{k}\in\{1,\dots,n\}$,
\[
S_{i_{k}}=\left(  \sum_{\widehat{i_{k}}}a_{\mathbf{i}}^{2}\right)  ^{1/2}.
\]
We intend to show that $\sum_{i_{k}}S_{i_{k}}^{\lambda}\leq M^{\lambda
}\left\Vert A\right\Vert ^{\lambda}$. We write
\begin{align*}
\sum_{i_{k}}S_{i_{k}}^{\lambda}  &  =\sum_{i_{k}}S_{i_{k}}^{\lambda-2}%
\sum_{\widehat{i_{k}}}a_{\mathbf{i}}^{2}\\
&  =\sum_{i_{l}}\sum_{\widehat{i_{l}}}\frac{a_{\mathbf{i}}^{2}}{S_{i_{k}%
}^{2-\lambda}}\\
&  =\sum_{i_{l}}\sum_{\widehat{i_{l}}}\frac{a_{\mathbf{i}}^{2/s}}{S_{i_{k}%
}^{2-\lambda}}a_{\mathbf{i}}^{2/s^{\ast}}%
\end{align*}
where $(s,s^{\ast})$ is a couple of conjugate exponents. We then apply
H\"{o}lder's inequality twice to get
\begin{align*}
\sum_{i_{k}}S_{i_{k}}^{\lambda}  &  \leq\sum_{i_{l}}\left(  \sum
_{\widehat{i_{l}}}\frac{a_{\mathbf{i}}^{2}}{S_{i_{k}}^{(2-\lambda)s}}\right)
^{\frac{1}{s}}\left(  \sum_{\widehat{i_{l}}}a_{\mathbf{i}}^{2}\right)
^{\frac{1}{s^{\ast}}}\\
&  \leq\left(  \sum_{i_{l}}\left(  \sum_{\widehat{i_{l}}}\frac{a_{\mathbf{i}%
}^{2}}{S_{i_{k}}^{(2-\lambda)s}}\right)  ^{\frac{t}{s}}\right)  ^{\frac{1}{t}%
}\left(  \sum_{i_{l}}\left(  \sum_{\widehat{i_{l}}}a_{\mathbf{i}}^{2}\right)
^{\frac{t^{\ast}}{s^{\ast}}}\right)  ^{\frac{1}{t^{\ast}}}.
\end{align*}
We then choose
\[
s=\frac{2-\lambda^{\prime}}{2-\lambda}\text{ and }\frac{t^{\ast}}{s^{\ast}%
}=\frac{\lambda}{2}%
\]
so that
\[
\frac{t}{s}=\frac{\lambda^{\prime}}{\lambda}.
\]
The inequality becomes
\begin{align*}
\sum_{i_{k}}S_{i_{k}}^{\lambda}  &  \leq\left(  \sum_{i_{l}}\left(
\sum_{\widehat{i_{l}}}\frac{a_{\mathbf{i}}^{2}}{S_{i_{k}}^{(2-\lambda^{\prime
})}}\right)  ^{\frac{\lambda}{\lambda^{\prime}}}\right)  ^{\frac{1}{t}}\left(
\sum_{i_{l}}\left(  \sum_{\widehat{i_{l}}}a_{\mathbf{i}}^{2}\right)
^{\frac{\lambda}{2}}\right)  ^{\frac{1}{t^{\ast}}}\\
&  \leq(M\Vert A\Vert)^{\frac{\lambda}{t^{\ast}}}\left(  \sum_{i_{l}}\left(
\sum_{\widehat{i_{l}}}\frac{a_{\mathbf{i}}^{2}}{S_{i_{k}}^{(2-\lambda^{\prime
})}}\right)  ^{\frac{\lambda}{\lambda^{\prime}}}\right)  ^{\frac{1}{t}}.
\end{align*}
It remains to control the last sum appearing on the right hand side of the
inequality. By (the converse of) H\"{o}lder's inequality, it is sufficient to
show that
\[
\sum_{i_{l}}\left(  \sum_{\widehat{i_{l}}}\frac{a_{\mathbf{i}}^{2}}{S_{i_{k}%
}^{(2-\lambda^{\prime})}}\right)  |y_{i_{l}}|\leq(M\Vert A\Vert)^{\lambda
^{\prime}}%
\]
for any $y\in B_{\ell_{\left(  \frac{\lambda}{\lambda^{\prime}}\right)
^{\ast}}^{n}}$. Since $\left(  \frac{\lambda}{\lambda^{\prime}}\right)
^{\ast}=\frac{p_{l}}{\lambda^{\prime}}$, this means that we want to prove
that
\[
\sum_{i_{l}}\sum_{\widehat{i_{l}}}\frac{a_{\mathbf{i}}^{2}}{S_{i_{k}%
}^{2-\lambda^{\prime}}}|x_{i_{l}}|^{\lambda^{\prime}}\leq(M\Vert
A\Vert)^{\lambda^{\prime}}%
\]
for any $x\in B_{\ell_{p_{l}}^{n}}$. Now,
\begin{align*}
\sum_{i_{l}}\sum_{\widehat{i_{l}}}\frac{a_{\mathbf{i}}^{2}}{S_{i_{k}%
}^{2-\lambda^{\prime}}}|x_{i_{l}}|^{\lambda^{\prime}}  &  =\sum_{i_{k}}%
\sum_{\widehat{i_{k}}}\frac{a_{\mathbf{i}}^{2}}{S_{i_{k}}^{2-\lambda^{\prime}%
}}|x_{i_{l}}|^{\lambda^{\prime}}\\
&  \leq\sum_{i_{k}}\left(  \sum_{\widehat{i_{k}}}\frac{a_{\mathbf{i}}^{2}%
}{S_{i_{k}}^{2}}\right)  ^{\frac{2-\lambda^{\prime}}{2}}\left(  \sum
_{\widehat{i_{k}}}a_{\mathbf{i}}^{2}|x_{i_{l}}|^{2}\right)  ^{\frac
{\lambda^{\prime}}{2}}%
\end{align*}
by H\"{o}lder's inequality with $\frac{2-\lambda^{\prime}}{2}+\frac
{\lambda^{\prime}}{2}=1$. To conclude, it remains to observe that
\[
\sum_{\widehat{i_{k}}}\frac{a_{\mathbf{i}}^{2}}{S_{i_{k}}^{2}}=\frac
{\sum_{\widehat{i_{k}}}a_{\mathbf{i}}^{2}}{S_{i_{k}}^{2}}=1
\]
and to use (\ref{EQ2}).

\medskip

It remains to deduce the proposition from $\mathcal{P}_{m-1}$. The argument is
exactly the same as we deduced $\mathcal{P}_{l}$ from $\mathcal{P}_{l-1}$,
except that we do not need the second (and more difficult) part, since there
are no $k>m$ in $\{1,\dots,m\}$. This explains why we just need
\[
\forall l\in\{1,\dots,m\},\ \sum_{k\neq l}\frac1{p_{k}}\leq\frac1r-\frac12
\]
and not
\[
\sum_{k=1}^{m} \frac1{p_{k}}\leq\frac1r-\frac12.
\]

\end{proof}

\begin{remark}
We cannot avoid the condition ``for any $l\geq1,\ \sum_{k\neq l}\frac1{p_{k}%
}\leq\frac1r-\frac12$'', see Section \ref{SECOPTIMALITYL2}.
\end{remark}

\section{First part of the proof of the main theorem}

We now prove that (2) implies (1) in Theorem \ref{THMMAIN}. Let $X=\ell_{s}$,
$Y=\ell_{q}$ and let $W$ be the Banach space of continuous $m$-linear forms
from $X_{p_{1}}\times\dots\times X_{p_{m}}$ to $\ell_{s}$. By the Bennett-Carl
inequalities (\cite{bennett, Carl}), the injection map $v:\ell_{s}%
\rightarrow\ell_{q}$ is $(r,1)$-summing with $\frac{1}{r}=\frac{1}{2}+\frac
{1}{s}-\frac{1}{q}.$ Then
\[
\lambda=\frac{1}{\frac{1}{2}+\frac{1}{s}-\frac{1}{q}-\left\vert {\frac
{1}{\mathbf{p}}}\right\vert }=\frac{1}{\frac{1}{r}-\left\vert {\frac
{1}{\mathbf{p}}}\right\vert }\leq2.
\]
Let finally $T$ defined on $W$ by $T(A)=(A(e_{i_{1}},\dots,e_{i_{m}}))$ and
$\mathbf{q}(k)=(2,\dots,2,\lambda,2,\dots,2)$. Applying Propositions
\ref{PROPL1L2} and \ref{PROPMINKO}, $T$ maps $W$ into $\ell_{\mathbf{q}(k)}$
for any $k\in\{1,\dots,m\}$. By our general interpolation procedure, $T$ maps
$X$ into $\ell_{\mathbf{q}}$ for any $\mathbf{q}=(q_{1},\dots,q_{m})\in
\lbrack\lambda,2]^{n}$ with
\[
\frac{1}{q_{1}}+\dots+\frac{1}{q_{m}}=\frac{m-1}{2}+\frac{1}{\lambda}=\frac
{m}{2}+\frac{1}{s}-\frac{1}{q}-\left\vert {\frac{1}{\mathbf{p}}}\right\vert .
\]

\begin{remark}
If we take care of the constant in the Bohnenblust-Hille inequality, our
method shows that it is valid with constant $C_{m}\leq\left(  \frac{2}%
{\sqrt{\pi}}\right)  ^{m-1}$ when $\mathbb{K}=\mathbb{C}$ or $C_{m}\leq
(\sqrt{2})^{m-1}$ when $\mathbb{K}=\mathbb{R}$. This constant comes from the
best known constant in the mixed $(\ell_{1},\ell_{2})$-Littlewood inequality.
In turn, this constant comes from the best constant in the Khintchine
$L^{1}-L^{2}$-inequality. However, we know that the best constant in the
Hille-Bohnenblust inequality is subpolynomial (see \cite{nun}). It would be
nice to know if the constant in the mixed $(\ell_{1},\ell_{2})$-Littlewood
inequality can also be chosen to be subpolynomial.
\end{remark}

\section{On the optimality}

\subsection{A Kahane-Salem-Zygmund inequality}

A way to prove that the exponent $\frac{2m}{m+1}$ is optimal in the
Bohnenblust-Hille inequality is to use the Kahane-Salem-Zygmund inequality,
which allows to control the infinite norm of random polynomials. We need two
variants of this inequality.

\begin{lemma}
\label{LEMKSZ} Let $d,n\geq1$, $p_{1},\dots,p_{d}\in[1,+\infty]^{d}$ and let,
for $p\geq1$,
\[
\alpha(p)=\left\{
\begin{array}
[c]{ll}%
\displaystyle \frac12-\frac1p & \text{ if }p\geq2\\
0 & \text{ otherwise.}%
\end{array}
\right.
\]
Then there exists a $d$-linear map $A:\lp1n\times\dots\ell_{p_{d}}^{n}%
\to\mathbb{C}$ of form
\[
A(z^{(1)},\dots,z^{(d)})=\sum_{i_{1},\dots,i_{d}}\pm z_{i_{1}}^{(1)}\cdots
z_{i_{d}}^{(d)}%
\]
such that
\[
\|A\|\leq C_{d} n^{\frac12+\alpha(p_{1})+\dots+\alpha(p_{d})}.
\]

\end{lemma}

The proof follows from a straightforward modification of \cite{BOAS}*{Theorem
4} and is therefore omitted.

We then also need a vector-valued version of this lemma.

\begin{lemma}
\label{LEMOPTI2} Let $d,n\geq1$, $s,p_{1},\dots,p_{d}\in[1,+\infty]^{d+1}$ and
let, for $p\geq1$,
\[
\alpha(p)=\left\{
\begin{array}
[c]{ll}%
\displaystyle \frac12-\frac1p & \text{ if }p\geq2\\
0 & \text{ otherwise.}%
\end{array}
\right.
\]
Then there exists a $d$-linear map $A:\lp1n\times\dots\ell_{p_{d}}^{n}%
\to\mathbb{\ell}_{p_{s}}^{n}$ of form
\[
A(z^{(1)},\dots,z^{(d)})=\sum_{i_{1},\dots,i_{d},i_{d+1}}\pm z_{i_{1}}%
^{(1)}\cdots z_{i_{d}}^{(d)}e_{i_{d+1}}%
\]
such that
\[
\|A\|\leq C_{d} n^{\frac12+\alpha(p_{1})+\dots+\alpha(p_{d})+\alpha(s^{*})}.
\]

\end{lemma}

\begin{proof}
This lemma is an easy consequence of the previous lemma. Indeed, there is an
isometric correspondence between $d$-linear maps $\ell_{p_{1}}^{n}\times
\dots\times\ell_{p_{d}}^{n}\rightarrow\ell_{s}^{n}$ and $(d+1)$-linear maps
$\ell_{p_{1}}^{n}\times\dots\times\ell_{p_{d}}^{n}\times\ell_{s^{\ast}}%
^{n}\rightarrow\mathbb{C}$. The correspondence is given by
\[
\left(  z\mapsto\sum_{i_{1},\dots,i_{d+1}}a_{i_{1},\dots,i_{d+1}}z_{i_{1}%
}^{(1)}\cdots z_{i_{d}}^{(d)}e_{i_{d+1}}\right)  \mapsto\left(  z\mapsto
\sum_{i_{1},\dots,i_{d+1}}a_{i_{1},\dots,i_{d+1}}z_{i_{1}}^{(1)}\cdots
z_{i_{d+1}}^{(d+1)}\right)  .
\]

\end{proof}

\subsection{Optimality of Theorems \ref{THMBHQ} and \ref{THMMAIN}}

We just prove the optimality of Condition (2) in Theorem \ref{THMMAIN}. Thus,
let $q_{1},\dots,q_{m}\in[\lambda,2]$ satisfying $(1)$ of Theorem
\ref{THMMAIN}. Let $A:\lp1n\times\dots\times\ell_{p_{m}}^{n}\to\ell_{s}^{n}$
be given by Lemma \ref{LEMOPTI2}. Then
\[
\|A\|\leq C_{m} n^{\frac12+\frac{m+1}2-\frac1{p_{1}}-\dots-\frac1{p_{m}}%
-\frac1{s^{*}}}=C_{m} n^{\frac{m}2-\left|  {\frac{1}{\mathbf{p}}}\right|
+\frac1s}%
\]
whereas, for any $i_{1},\dots,i_{m}$,
\[
\|A(e_{i_{1}},\dots,e_{i_{m}})\|_{q}=n^{1/q}.
\]
Then,
\[
\left(  \sum_{i_{m}=1}^{n}\|A(e_{i_{1}},\dots,e_{i_{m}})\|_{q}^{q_{m}}\right)
^{\frac{q_{m-1}}{q_{m}}}=n^{\frac{q_{m-1}}{q_{m}}+\frac{q_{m-1}}q}.
\]
It is then easy to show by induction that
\[
\left(  {\textstyle\sum\limits_{j_{1}=1}^{\infty}}\left(  {\textstyle\sum
\limits_{j_{2}=1}^{\infty}}\left(  ...\left(  {\textstyle\sum\limits_{j_{m-1}%
=1}^{\infty}}\left(  {\textstyle\sum\limits_{j_{m}=1}^{\infty}} \|A\left(
e_{j_{1}},...,e_{j_{m}}\right)  \|_{q} ^{q_{m}}\right)  ^{\frac{q_{m-1}}%
{q_{m}}}\right)  ^{\frac{q_{m-2}}{q_{m-1}}}\cdots\right)  ^{\frac{q_{2}}%
{q_{3}}}\right)  ^{\frac{q_{1}}{q_{2}}}\right)  ^{\frac{1}{q_{1}}}%
=n^{\frac1q+\frac1{q_{1}}+\dots+\frac1{q_{m}}}.
\]
For $(1)$ of Theorem \ref{THMMAIN} to be true, it is necessary that
\[
\frac1q+\dots+\frac1{q_{m}}\leq\frac{m}2-\left|  {\frac{1}{\mathbf{p}}%
}\right|  +\frac1s.
\]

\subsection{The case $m=2$}

We now study more precisely the optimality in the Bohnenblust-Hille when $m=2$
and $\mathbb{K}=\mathbb{R}$.

\begin{theorem}
\label{67} Let $1 \leq p,q \leq2$ with $\frac{1}{p}+\frac{1}{q}\leq\frac{3}%
{2}.$ Then for every continuous bilinear form $A:c_{0}\times c_{0}%
\rightarrow{\mathbb{R}}$
\[
\left(  \sum_{k=1}^{\infty} \left(  \sum_{j=1}^{\infty} \left|  A\left(
e_{k},e_{j}\right)  \right|  ^{q} \right)  ^{p/q}\right)  ^{1/p}\leq
2^{\frac1p+\frac1q-1}\Vert A\Vert
\]
Moreover the constant $2^{\frac1p+\frac1q-1}$ is optimal.
\end{theorem}

\begin{proof}
If $\frac1p+\frac1q=\frac32$, it has already been observed that the inequality
holds true. Suppose now $1<p,q<2$ and
\[
\frac{1}{p}+\frac{1}{q}<\frac{3}{2}.
\]
Let%
\[
\theta_{0}:=2\left(  \frac{1}{p}+\frac{1}{q}-1\right)  .
\]
Note that $\theta_{0}\in\left(  0,1\right)  $ and let $p_{0},q_{0}$ be defined
by%
\begin{align*}
p_{0}  &  :=\frac{4p+4q-4pq}{2p+4q-3pq},\\
q_{0}  &  :=\frac{4p+4q-4pq}{4p+2q-3pq}.
\end{align*}

We thus have $1<p_{0},q_{0}<2$ and
\[
\frac{1}{p_{0}}+\frac{1}{q_{0}}=\frac{3}{2}.
\]
So, from Part 1 of the proof we have%
\[
\left(  \sum_{k=1}^{\infty}\left(  \sum_{j=1}^{\infty}|A(e_{k},e_{j})|^{q_{0}%
}\right)  ^{p_{0}/q_{0}}\right)  ^{1/p_{0}}\leq\sqrt{2}\Vert A\Vert
\]
for all continuous bilinear forms $A:c_{0}\times c_{0}\rightarrow{\mathbb{R}}%
$. On the other hand it is well known that%
\[
\left(  \sum_{k=1}^{\infty}\left(  \sum_{j=1}^{\infty}|A(e_{k},e_{j}%
)|^{2}\right)  \right)  ^{1/2}\leq\Vert A\Vert.
\]
Since%
\[
\left\{
\begin{array}
[c]{c}%
\frac{1}{p}=\frac{\theta_{0}}{p_{0}}+\frac{1-\theta_{0}}{2},\\
\frac{1}{q}=\frac{\theta_{0}}{q_{0}}+\frac{1-\theta_{0}}{2},
\end{array}
\right.
\]
from complex interpolation we have%
\begin{align*}
\left(  \sum_{k=1}^{\infty}\left(  \sum_{j=1}^{\infty}|A(e_{k},e_{j}%
)|^{q}\right)  ^{p/q}\right)  ^{1/p}  &  \leq\left(  \sqrt{2}\right)
^{\theta_{0}}\Vert A\Vert\\
&  =2^{\frac{1}{p}+\frac{1}{q}-1}\Vert A\Vert.
\end{align*}
The case $p,q\in\left\{  1,2\right\}  $ with $\frac{1}{p}+\frac{1}{q}<\frac
{3}{2}$ can be verified straightforwardly.

On the other hand, using ideas from \cite{diniz}, consider the bilinear form
\[
A(x,y)=x_{1}y_{1}+x_{1}y_{2}+x_{2}y_{1}-x_{2}y_{2}%
\]
which satisfies $\|A\|=2$. Then $A$ verifies the equality
\[
\left(  \sum_{k=1}^{2} \left(  \sum_{j=1}^{2} \left|  A\left(  e_{k}%
,e_{j}\right)  \right|  ^{q} \right)  ^{p/q}\right)  ^{1/p}\leq2^{\frac
1p+\frac1q-1}\Vert A\Vert
\]
which shows that the constant $2^{\frac1p+\frac1q-1}$ cannot be improved.
\end{proof}

\begin{remark}
In the complex case, we obtain that for any continuous bilinear form
$A:c_{0}\times c_{0}\to\mathbb{C}$,
\[
\left(  \sum_{k=1}^{\infty} \left(  \sum_{j=1}^{\infty} \left|  A\left(
e_{k},e_{j}\right)  \right|  ^{q} \right)  ^{p/q}\right)  ^{1/p}\leq\left(
\frac2{\sqrt\pi}\right)  ^{\frac1p+\frac1q-1}\Vert A\Vert.
\]
We have no information on the optimality of this estimate.
\end{remark}

\subsection{On the conditions of the mixed $(\ell_{\lambda},\ell_{2})$-norm
inequality}

\label{SECOPTIMALITYL2} We now show that we cannot avoid a condition like
\[
\forall l\in\{1,\dots,m\},\ \sum_{k\neq l}\frac{1}{p_{k}}\leq\frac{1}{r}%
-\frac{1}{2}%
\]
in the statement of Proposition \ref{PROPL1L2}. Indeed, let $X=Y=\mathbb{C}$
and $v=Id$, such that $r=1$. Assume on the contrary that $\sum_{k=1}%
^{m-1}\frac{1}{p_{k}}>\frac{1}{2}$. Let $F:\ell_{p_{1}}^{n}\times\dots
\times\ell_{p_{m-1}}^{n}\rightarrow\mathbb{C}$ be given by Lemma \ref{LEMKSZ},
so that $\Vert F\Vert\leq Cn^{\frac{m}{2}-\frac{1}{p_{1}}-\dots-\frac
{1}{p_{m-1}}}$. Let us define $A:\lp1n\times\dots\times\ell_{p_{m}}%
^{n}\rightarrow\mathbb{C}$ by
\[
A(z^{(1)},\dots,z^{(m)})=z_{1}^{(m)}F(z^{(1)},\dots,z^{(m-1)}),
\]
so that $\Vert A\Vert=\Vert F\Vert$. Then, for any $\theta>0$,
\[
\left(  \sum_{i_{m}}\left(  \sum_{\widehat{i_{m}}}|A(e_{i_{1}},\dots
,e_{i_{m-1}},e_{i_{m}})|^{2}\right)  ^{\theta/2}\right)  ^{1/\theta}%
\geq\left(  \sum_{\widehat{i_{m}}}|A(e_{i_{1}},\dots,e_{i_{m-1}},e_{i_{1}%
})|^{2}\right)  ^{1/2}.
\]
Now, we cannot have
\[
\left(  \sum_{\widehat{i_{m}}}|A(e_{i_{1}},\dots,e_{i_{m-1}},e_{i_{1}}%
)|^{2}\right)  ^{1/2}\leq C\Vert A\Vert
\]
since the left hand side is equal to $n^{\frac{m-1}{2}}$ and the right hand side
behaves like $n^{\frac{m-1}{2}+\frac{1}{2}-\frac{1}{p_{1}}-\dots-\frac
{1}{p_{m-1}}}$.

\bigskip

\noindent\textbf{Acknowledgements}. The authors would like to thank Prof.
Fernando Cobos for fruitful conversations and his interest in this note. Part
of this paper was written when the third named author was visiting Professors
C. Barroso and E. Teixeira at the Department of Mathematics of Universidade
Federal do Cear\'{a}, Brazil; he is very grateful for their hospitality.

%%%%%%%%%%
%REFERENCES %
%%%%%%%%%%

\end{document}